\documentclass{amsart}

\usepackage{amssymb,latexsym,amsmath,amsfonts,amsthm}

\usepackage{color}
\usepackage{subfigure}

\usepackage[cp1250]{inputenc}

\newcommand{\nN}{\mathbb{N}}

\newcommand{\rR}{\mathbb{R}}

\newtheorem{theorem}{Theorem}[section]

\newtheorem{corollary}[theorem]{Corollary}

\newtheorem{definition}[theorem]{Definition}

\newtheorem{lemma}[theorem]{Lemma}

\newtheorem{remark}[theorem]{Remark}

\begin{document}
%%% ENTER TITLE
\title{A convexity in $\rR^2$ with river metric}
%%% AUTHOR(S) FULL NAMES, AND EMAIL ADDRESSES
\author[Oki\v ci\' c]{Nermin Oki\v ci\' c} %% Please write ful names, avoid initials
\email{nermin.okicic@untz.ba}

\author[Reki\' c-Vukovi\' c]{Amra Reki\' c-Vukovi\' c}
\email{amra.rekic@untz.ba}
%%%
%%%

\begin{abstract}
In this paper we consider the space $\rR^2$ with the river metric $d^*$ and different types of  convexity of this space. We define $W$-convex structure in $(\rR^2,d^*)$ and we give the complete characterization of the convex sets in this space. We consider some measures of noncompactness  and we give the moduli of noncompactness for considered measures on this space.
\end{abstract}

\maketitle

%%%%%%%%%%%%%%%%%%%%%%%%%%%
%%%%%%%%%%%%%%%%%%%%%%%%%%%%%%%%%%%%%%%%%%%%%%%%%%%%%%%%%%%%%%%%%%%%%%%%%%%%%%
\section{Introduction}
 The space $\rR^2$ with the river metric $d^*$ is not normed space. It is not even a linear metric space. This can cause certain difficulties in observing spaces like this, because the definitions of many terms must be adapted to these spaces. The concept of the convexity structures is not limited on the linear vector spaces, but also can be extended for some interesting notions of the convex structures for the metric spaces, even for the topological spaces. Takahashi \cite{article.7} introduced the notion of the convex structure on the metric spaces which is known as $W$-convex structure and it is generalization of the convexity, which give us the standard term of the convex sets in normed spaces. Menger \cite{article.8} defined the concept of the convexity in metric spaces called Menger convexity which is the generalization of Takahashi convex structure. Bryant \cite{article.3} considered the notion of the (metrically) convex metric space and (metric) segment space, pointing that every segment space is metrically convex, but in complete metric spaces these two concepts coincide, \cite{article.9}. On the other hand, Foertsch \cite{article.1} compares two different notions of convexity of geodesic metric spaces namely
(strict/uniform) ball convexity and (strict/uniform) distance convexity. Strict convexity is usually one of the main subjects in studying Banach spaces and it is known that metric lines are unique in Banach spaces, \cite{article.10}. But, this doesn't have to hold in the pure metric spaces. There are complete convex metric spaces, externally convex metric spaces in which these concepts of unique metric lines and strict convexity, defined in terms of metric spaces, are not equivalent. It was shown that
unique metric lines (metric segments have unique prolongations) and strict
convexity are equivalent in a larger class of spaces, \cite{article.5}, \cite{article.6}. Narang and Tejpal \cite{article.4} considered spaces in which metric segments have unique prolongations and gave results on the best approximation in such metric spaces, also considering $M$-spaces introduced by Khalil \cite{article.5} and their relations with strict convexity. Khamsi \cite{article.14} considered nearly uniformly convexity in metric spaces and Takahashi \cite{article.13} introduced a notion of uniform convexity in convex metric spaces. \\
In the following, we will observe the set $\mathbb{R}^2$ with so-called river metric defined by
\[
d^*(v_1,v_2)=\left\{
\begin{array}{ccc}
  |y_1-y_2| & , & x_1=x_2 ,\\
  |y_1|+|y_2|+|x_1-x_2| & , & x_1\neq x_2, \\
\end{array} \right .
\]
where $v_1=(x_1,y_1)$, $v_2=(x_2,y_2)\in \rR^2$. We will denote the open ball by $B((x_0,y_0),r)=\{(x,y)\in \rR^2 \ | \ d^*((x,y),(x_0,y_0))<r\}$, where $v=(x_0,y_0)\in \rR^2$ is the center of the ball, and $r>0$ is the radius. With $\overline{B}((x_0,y_0),r)$ we denote closed ball and generally with $\overline{A}$ the closure of the set $A$. \\
Following lemma describes the metric segment in $(\rR^2,d^*)$.

\begin{lemma}[Lemma 2.1, \cite{article.11}]\label{MetrickiSegment}
The metric segment between points $v_1=(x_1,y_1), \ v_2=(x_2,y_2)\in \rR^2$ with the river metric $d^*$ is the set
\begin{align*}
[v_1,v_2]= & \{(x_1,a)\in \rR^2 \ | \ a\in [0,y_1] \ (\textrm{or }a\in [y_1,0])\} \\
& \cup \{(x_2,a)\in \rR^2 \ | \ a\in [0,y_2] \ (\textrm{or }a\in [y_2,0])\}\cup \{(b,0)\in \rR^2 \ | \ x_1\leq b\leq x_2\},
\end{align*}
for $x_1< x_2$, or the set
\[
[v_1,v_2]=\{(x_1,a)\in \rR^2 \ | \ y_1\leq a\leq y_2\},
\]
for $x_1=x_2$ and $y_1\leq y_2$.
\end{lemma}
%%%%%%%%%%%%%%%%%%%%%%%%%%%%%%%%%%%%%%%%%%%%%%%%%%%%%%%%%%%%%%%%%%%%%%%%%%%%%%%%%%%%%%%%%%%%%%%%%%%%%%%%%

%%%%%%%%%%%%%%%%%%%%%%%%%%%%%%%%%%%%%%%%%%%%%%%%%%%%%%%%%%%%%%%%%%%%%%%%%%%%%%%%%%%%%%%%%%%%%%%%%%%%%%%%%%%%%%%%%%%%%%%%%%%%%%%%%%%%%%%%%%%%%
\section{Convex sets in the space $(\rR^2,d^*)$}

In \cite{article.7} Takahashi defines a convex structure on a metric space which allows the definition of term of convex sets in this space.

\begin{definition}[\cite{article.7}]\label{Definicija_Wstruktura}
A mapping $W:X^2\times [0,1]\rightarrow X$ is said to be a convex structure on $X$ if for all $x,y,u \in X$ and $\lambda \in [0,1]$
\begin{equation}\label{DW_konveksnastruktura}
d(u,W(x,y,\lambda))\leq \lambda d(u,x)+(1-\lambda)d(u,y) .
\end{equation}
We call a metric space with a convex structure $W$, a $W$-convex metric space.
\end{definition}
The mapping $W:X\times X\times [0,1] \rightarrow X$ given by
\[
W(x,y,\lambda)=\lambda x+(1-\lambda)y
\]
defines convex structure on arbitrary normed space $X$. Therefore, any normed space is $W$-convex space.

According to the previous definition we will describe the convex structure on $(\rR^2, d^*)$. Let us define the function $W:\rR^2\times \rR^2\times [0,1]\rightarrow \rR^2$ such that for $v_1,v_2\in \rR^2$ and $\lambda \in [0,1]$
\[
W(v_1,v_2,\lambda)=z=(z_x,z_y) ,
\]
where
\begin{equation}\label{uslov1}
d^*(z,v_1)=(1-\lambda)d^*(v_1,v_2),
\end{equation}
\begin{equation}\label{uslov2}
d^*(z,v_2)=\lambda d^*(v_1,v_2) .
\end{equation}

We will prove that $W$ is the convex structure. Indeed, let $v_1=(x_1,y_1)$, $v_2=(x_2,y_2)$ be arbitrary and, without loss of generality, $0\leq x_1<x_2$. The point $z=W(v_1,v_2,\lambda)$ has three different positions in $\rR^2$, depending on $\lambda \in [0,1]$, as follows. \\
(a) \ \  Since
\begin{align*}
    d^*(v_1,z)\leq |y_1| & \Longleftrightarrow  (1-\lambda)d^*(v_1,v_2)\leq |y_1| \\
     & \Longleftrightarrow  \lambda \geq 1-\displaystyle \frac{|y_1|}{d^*(v_1,v_2)} ,
  \end{align*}
we conclude that $z=(x_1,z_y)$. Using (\ref{uslov1}) and (\ref{uslov2}) we have $\lambda d^*(z,v_1)=(1-\lambda)d^*(z,v_2),$ i.e.
\begin{equation}\label{jednakost3}
\lambda |y_1-z_y|=(1-\lambda )|z_y|+(1-\lambda )|y_2|+(1-\lambda)|x_1-x_2| .
\end{equation}
If $y_1\geq 0$, then $0\leq z_y\leq y_1$, and using (\ref{jednakost3}) we get
\[
z_y=\lambda y_1-(1-\lambda )|y_2|-(1-\lambda)|x_1-x_2| \ .
\]
If $y_1<0$, then $y_1<z_y\leq 0$, and using (\ref{jednakost3}) gives us
\[
z_y=\lambda y_1+(1-\lambda)|y_2|+(1-\lambda )|x_1-x_2| .
\]
(b) \ \ Since
\begin{align*}
    |y_1|<d^*(v_1,z)\leq |y_1|+|x_1-x_2| & \Longleftrightarrow   |y_1|<(1-\lambda)d^*(v_1,v_2)\leq |y_1|+|x_1-x_2|\\
     & \Longleftrightarrow  1-\frac{|y_1|+|x_1-x_2|}{d^*(v_1,v_2)}\leq \lambda <1- \frac{|y_1|}{d^*(v_1,v_2)} ,
\end{align*}
we have $z=(z_x,0)$. Using conditions (\ref{uslov1}) and (\ref{uslov2}) we have $\lambda d^*(v_1,z)=(1-\lambda)d^*(v_2,z)$, i.e. $\lambda |y_1|+\lambda |z_x-x_1|=(1-\lambda)|y_2|+(1-\lambda)|z_x-x_2|$. Since $x_1\leq z_x\leq x_2$ we conclude
\[
z_x=-\lambda |y_1|+(1-\lambda)|y_2|+\lambda x_1+(1-\lambda)x_2 .
\]
(c) \ \ From
\begin{align*}
    |y_1|+|x_1-x_2|<d^*(v_1,z)\leq d^*(v_1,v_2) & \Longleftrightarrow  |y_1|+|x_1-x_2|\leq (1-\lambda)d^*(v_1,v_2)\leq d^*(v_1,v_2) \\
     & \Longleftrightarrow  0\leq \lambda \leq 1-\displaystyle \frac{|y_1|+|x_1-x_2|}{d^*(v_1,v_2)} ,
\end{align*}
we have $z=(x_2,z_y)$. The equality $\lambda d^*(v_1,z)=(1-\lambda)d^*(v_2,z)$ is equivalent with
\begin{equation}\label{jednakost4}
\lambda |y_1|+\lambda |z_y|+\lambda |x_1-x_2|=(1-\lambda)|y_2-z_y| .
\end{equation}
If $y_2\geq 0$ and $0\leq z_y\leq y_2$, then form the equality (\ref{jednakost4}) we have
\[
z_y=-\lambda |y_1|+(1-\lambda)y_2-\lambda |x_1-x_2|.
\]
If $y_2<0$ and $y_2\leq z_y\leq 0$ and using (\ref{jednakost4}) we have
\[
z_y=\lambda |y_1|+(1-\lambda)y_2+\lambda |x_1-x_2| .
\]
Now, let us check the accuracy of the inequality (\ref{DW_konveksnastruktura}) for our defined function $W$. Let $v=(x,y)$ be arbitrary and we denote
\begin{align*}
&A=d^*(v,z)=|y|+|z_y|+|x-x_z|, \\
&B=\lambda d^*(v,v_1)+(1-\lambda)d^*(v_1,v_2)=|y|+\lambda |y_1|+(1-\lambda)|y_2|+\lambda|x-x_1|+(1-\lambda)|x-x_2|.
\end{align*}
Depending on $\lambda\in [0,1]$, we will consider the following cases. \\
(a) \ \ Let $1-\displaystyle \frac{|y_1|}{d^*(v_1,v_2)}\leq \lambda \leq 1$. This means that $z=(x_1,z_y)$.
If $z_y=\lambda y_1-(1-\lambda)|y_2|-(1-\lambda)|x_1-x_2|$, then
\begin{align*}
    A & =  |y|+\lambda y_1-(1-\lambda)|y_2|-(1-\lambda)|x_1-x_2|+|x-x_1| \\
     & \leq  |y|+\lambda |y_1|+(1-\lambda)|y_2|-(1-\lambda)|x_1-x_2|+|x-x_1|  .
\end{align*}
Since $|x-x_1|-(1-\lambda)|x_1-x_2|\leq \lambda |x-x_1|+(1-\lambda)|x-x_2|$, it holds $A\leq B$. \\
If $z_y=\lambda y_1+(1-\lambda)|y_2|+(1-\lambda)|x_1-x_2|$, then
\begin{align*}
    A & =  |y|-\lambda y_1-(1-\lambda)|y_2|-(1-\lambda)|x_1-x_2|+|x-x_1| \\
     & \leq   |y|+\lambda y_1+(1-\lambda)|y_2|-(1-\lambda)|x_1-x_2|+|x-x_1|  .
\end{align*}
Since $|x-x_1|-(1-\lambda)|x_1-x_2|\leq \lambda |x-x_1|+(1-\lambda)|x-x_2|$, we have
\[
A\leq |y|+\lambda|y_1|+(1-\lambda)|y_2|+\lambda |x-x_1|+(1-\lambda)|x-x_2|=B .
\]

(b) \ \ Let $1-\displaystyle \frac{|y_1|+|x_1-x_2|}{d^*(v_1,v_2)}\leq \lambda <1-\displaystyle \frac{|y_1|}{d^*(v_1,v_2)}$, that is $z=(z_x,0)$. Then $A=|y|+|x-z_x|$. If $x>z_x$, we have
\begin{align*}
    A & =  |y|+x+\lambda^|y_1|-(1-\lambda)|y_2|-\lambda x_1-(1-\lambda)x_2 \\
     & \leq  |y|+\lambda|y_1|+(1-\lambda)|y_2|+x-\lambda x_1-(1-\lambda)x_2  .
\end{align*}
Since
\[
x-\lambda x_1-(1-\lambda)x_2=\lambda (x-x_1)+(1-\lambda)(x-x_2)\leq \lambda |x-x_1|+(1-\lambda)|x-x_2|,
\]
we have $A\leq B$. \\
Let $x<z_x$. Then $A=|y|+z_x-x$ and
\begin{align*}
    A & =  |y|+\lambda|y_1|+(1-\lambda)|y_2|+(1-\lambda)x_2+\lambda x_1-x \\
     & \leq  |y|+\lambda |y_1|+(1-\lambda)|y_2|+\lambda (x_1-x)+(1-\lambda)(x_2-x) ,
\end{align*}
i.e. $A\leq |y|+\lambda |y_1|+(1-\lambda)|y_2|+\lambda |x_1-x|+(1-\lambda)|x_2-x|=B$.

\noindent (c) \ \ Let $0\leq  \lambda <1-\displaystyle \frac{|y_1|+|x_1-x_2|}{d^*(v_1,v_2)}$, that is $z=(x_2,z_y)$. If $z_y=-\lambda |y_1|+(1-\lambda )y_2-\lambda |x_1-x_2|$, then
\begin{align*}
    A & =  |y|-\lambda |y_1|+(1-\lambda)y_2-\lambda |x_1-x_2|+|x-x_2| \\
     & \leq  |y|+\lambda|y_1|+(1-\lambda)|y_2|-\lambda|x_1-x_2|+|x-x_2|  .
\end{align*}
Since $-\lambda|x_1-x_2|+|x-x_2|\leq \lambda |x-x_1|+(1-\lambda)|x-x_2|$, it holds
\[
A\leq |y|+\lambda|y_1|+(1-\lambda )|y_2|+\lambda |x-x_1|+(1-\lambda)|x-x_2|=B.
\]
If $z_y=\lambda |y_1|+(1-\lambda)y_2+\lambda|x_1-x_2|$, then
\begin{align*}
    A & =  |y|-\lambda |y_1|-(1-\lambda)y_2-\lambda |x_1-x_2|+|x-x_2| \\
     & \leq  |y|+\lambda |y_1|+(1-\lambda)|y_2|-\lambda|x_1-x_2|+|x-x_2|.
\end{align*}

This means that $A\leq B$. Indeed, the function $W$ is convex structure defined on $\rR^*$. Therefore, by the Definition \ref{Definicija_Wstruktura}, $(\rR^2,d^*)$ is $W$-convex metric space.

\begin{definition}[\cite{article.7}]
A nonempty subset $K$ of a $W$-convex metric space $(X,d)$ is said to be convex if $W(x,y,\lambda)\in K$ for every $x,y\in K$ and $\lambda \in [0,1]$.
\end{definition}

For any two points $x,y\in X$ the set
\[
\{ z\in X: \ d(x,y)=d(x,z)+d(z,y)\},
\]
is called a metric segment and is denoted by $[x,y]$, i.e.
\[
[x,y]=\{ z\in X: \ d(x,y)=d(x,z)+d(z,y)\} .
\]
Since for $x,y\in X$ and $\lambda \in [0,1]$ we have
\[
d(x,y)=d(x,W(x,y,\lambda))+d(W(x,y,\lambda),y) ,
\]
we can think about the metric segment $[x,y]$ as of the set of the values of the mapping $W$, for all $\lambda \in [0,1]$ and fixed $x,y \in X$, that is, for any two points $x,y\in X$, metric segment is the set
\[
[x,y]=\{ W(x,y,\lambda): \ \lambda \in [0,1]\} .
\]

\begin{remark}
A nonempty subset $K$ of a $W$-convex metric space $(X,d)$ is convex set if $[x,y]\subset K$ for all $x,y\in K$.
\end{remark}

With the following statement, we give the complete characterization of convex sets in the space $\rR^2$ with the river metric.

\begin{theorem}\label{Tkonveksniskupovi}
Let $D$ be a set that is not a ball in $(\rR^2,d^*)$. The set $D$ is convex if and only if for every point $(x,y)$ in $D$, the metric segment $[(x,0),(x,y)]$, is a subset of $D$ and $[ a,b]\times \{0\}\subset D$, where $a=\inf\{x \ | \ v=(x,y)\in D\}$ and $\displaystyle b=\sup\{x \ | \ v=(x,y)\in D\}$.
\end{theorem}
\begin{proof}
Let $D$ be a convex set in $(\rR^2,d^*)$ that is not a ball and let $(x,y)$ be an arbitrary point in $D$. Suppose $D=\{(x,a) \ | \ y_1\leq a \leq y_2\}$. If it were $y_1>0$, then the set $D$ would be the ball $B\left(\displaystyle \frac{y_2+y_1}{2},\displaystyle \frac{y_2-y_1}{2}\right)$, which we assumed that the set $D$ is not. Therefore, $y_1\leq 0$. Then the point $(x,0)$ belongs to the set $D$, and because of the convexity of the set $D$, $[(x,0),(x,y)]\subseteq D$ holds.\\
Suppose that the set $D$ also contains the point $(x_1,y_1)$, where $x_1\neq x$. Due to the convexity of the set $D$, $[(x,y),(x_1,y_1)]\subseteq D$ holds. According to Lemma \ref{MetrickiSegment} $[(x,y),(x_1,y_1)]=M_1\cup M_2 \cup M_3$, where $M_1=[(x_1,0),(x_1,y_1)]$, $M_2=[(x_2,0),(x_2,y_2)]$ and $M_3=[(x_1,0),(x_2,0)]$ are metric segments and it is obvious that $M_1\subset D$, that is $[(x,0 ),(x,y)]\subset D$.

Now let $D$ be a set that is not a ball and let for every point $(x,y)$ in $D$ be $[(x,0),(x,y)]\subset D$. Let $v_1=(x_1,y_1)$ and $v_2=(x_2,y_2)$ be arbitrary points in $D$ and let $x_1<x_2$. By assumption $M_1=[(x_1,0),(x_1,y_1)] \subset D$ and $M_2=[(x_2,0),(x_2,y_2)] \subset D$. As $x_1,x_2\in [a,b]$, where $a=\inf\{x \ | \ v=(x,y)\in D\}$ and $\displaystyle b=\sup\{x \ | \ v=(x,y)\in D\}$, it is clear that $M_3=[(x_1,0),(x_2,0)]\subseteq [a,b]\times \{0\}$ holds , that is $M_3\subseteq D$. So $M_1\cup M_2 \cup M_3=[v_1,v_2]\subseteq D$, and the set $D$ is convex.
\end{proof}

The next statement is a direct consequence of the above statement, and it gives us a special class of convex sets.

\begin{corollary}\label{KonveksniSkupovi}
Sets of the form $K=[a,b]\times [c,d]$, where $a,b,c,d\in \rR$, $c\leq 0$ and $d\geq 0$ are convex.
\end{corollary}

Special forms of convex sets according to the above corollary are sets of the form $[a,b]\times \{0\}$ and $\{a\}\times [c,d]$ ($a,b,c,d \in \rR$).\\
In the general case, the intersection of convex sets is again a convex set, while this is not the case for the union.

\begin{theorem}
Let $A,B\subset \rR^2$ be convex sets such that $A\cap B\neq \varnothing$. Then $A\cup B$ is a convex set.
\end{theorem}
\begin{proof}
Let $A$ and $B$ be convex sets in $\rR^2$ with river metric and let $A\cap B\neq \varnothing$. Let $v_1=(x_1,y_1)$ and $v_2=(x_2,y_2)$ be arbitrary elements from $A\cup B \subset \rR^2$. If $v_1,v_2 \in A$ or $v_1,v_2\in B$, it is clear due to the convexity of the sets that it will be $[v_1,v_2]\subset A\cup B$. So, suppose that $v_1\in A$ and $v_2\in B$. Based on Lemma \ref{MetrickiSegment} we have $[v_1,v_2]=M_1\cup M_2\cup M_3$,
where are $M_1=\{(x_1,a) \ | \ a\in [0,y_1] (\textrm{or } a\in [y_1,0])\}$, $M_2=\{(x_2,a) \ | \ a\in [0,y_2] (\textrm{or } a\in [y_2,0])\}$ and $M_3=\{(b,0) \ | \ x_1\leq b\leq x_2\}$. Because of the convexity of the set $A$ that is $M_1\subseteq A$, and analogously because of the convexity of the set $B$ that is $M_2\subseteq B$. Let now $v^*=(x^*,y^*)\in A\cap B$ arbitrary. Again, due to the convexity of the set $A$, $\{(b,0) \ | \ x_1\leq b\leq x^*\}\subseteq A$ and $\{(b,0) \ | \ x^*\leq b\leq x_2\}\subseteq B$. Now we have
\[
\{(b,0) \ | \ x_1\leq b\leq x^*\}\cup \{(b,0) \ | \ x^*\leq b\leq x_2\}=M_3 \subseteq A\cup B.
\]
Thus, $[v_1,v_2]\subseteq A \cup B$, and $A \cup B$ is a convex set.
\end{proof}

There are other ways haw to define the notion of the convexity on a metric space. For example,  Menger in \cite{article.8} introduced so-called Menger convexity of a metric space.

\begin{definition}[\cite{article.8}]\label{Menger konveksnost1}
A metric space $(X, d)$ is said to be Menger convex metric
space if for all $x, y\in X$, $x\neq y$, there is a point $z\in X$, $z\neq x$, $z\neq y$, such that
\[
d(x, y) = d(x, z) + d(z, y).
\]
\end{definition}
Menger convexity is also defined by the following definition which is equivalent to the Definition \ref{Menger konveksnost1}.
\begin{definition}[\cite{article.12}]
A metric space $(X, d)$ is said to be a Menger convex metric space if for all $x,y\in X$, $x\neq y$, $0\leq r\leq d(x,y)$ it holds
\[
B[x,r]\cap B[y,d(x,y)-r]\neq \varnothing .
\]
\end{definition}

$W$-convex metric spaces are Menger convex (\cite{article.16}), which in our case means that the metric space $(\rR^2,d^*)$ is Menger convex metric space.

%%%%%%%%%%%%%%%%%%%%%%%%%%%%%%%%%%%%%%%%%%%%%%%%%%%%%%%%%%%%%%%%%%%%%%%%%%%%%%%%%%%%%%%
%%%%%%%%%%%%%%%%%%%%%%%%%%%%%%%%%%%%%%%%%%%%%%%%%%%%%%%%%%%%%%%%%%%%%%%%%%%%%%%%%%%%%%%
%%%%%%%%%%%%%%%%%%%%%%%%%%%%%%%%%%%%%%%%%%%%%%%%%%%%%%%%%%%%%%%%%%%%%%%%%%%%%%%%%%%%%%%
\section{Different types of convexity of the space $(\rR^2,d^*)$}
\begin{definition}[\cite{article.3}]
A metric space $(X, d)$ is called metrically convex if for any two distinct points $v_1, v_2 \in X$ and for some positive real numbers $a, b$
such that $d(v_1, v_2) = a + b$ there exists a point $v \in X$ such that $d(v_1, v) = a$ and $d(v,v_2) = b$. If such a point exists for any decomposition $d(v_1,v_2) = a + b$, $a, b> 0$, then $X$ is called totally convex.
\end{definition}

It is not difficult to see that the above definition of metrically convexity of the space is equivalent to the fact that for every two dissimilar points $x,y\in X$, there is a point $z\in X$, $z\neq x,y$, such that $d(x,y)=d(x,z)+d(z,y)$ (\cite{article.3}). In a complete metric space, the concepts of metric convexity and segment space are equivalent, and $(R^2, d^*)$ is a metrically convex space.

\begin{theorem}
$(\rR^2,d^*)$ is totally convex metric space.
\end{theorem}
\begin{proof}
Let $v_1=(x_1,y_1)$ and $v_2=(x_2,y_2)$ be different points in $\rR^2$ and let $a,b>0$ be arbitrary, such that $d^*(v_1,v_2)=a+b$. This means that  $0<a,b< d^*(v_1,v_2)$. We consider three cases.\\
Let $0<a\leq |y_1|$. Let us choose $v=(x_1,y_1-a)$ if $y_1>0$ or $v=(x_1,y_1+a)$ if $y_1<0$. Then
\[
d^*(v_1,v)=|y_1-(y_1-a)|=a=|y_1-(y_1+a)|.
\]
Since $v$ is the point of the metric segment $[v_1,v_2]$, i.e. $d^*(v_1,v)+d^*(v,v_2)=d^*(v_1,v_2)$, it is clear that $d^*(v,v_2)=b$.

\noindent Now let $|y_1|<a\leq |y_1|+|x_1-x_2|$. Let us choose the point $v=(a-|y_1|+x_1,0)$. Therefore
\[
d^*(v_1,v)=|y_1|+|(a-|y_1|+x_1)-x_1|=|y_1|+|a-|y_1||=|y_1|+a-|y_1|=a,
\]
and again due to the fact that $v$ is the point of the metric segment, we have $d^*(v,v_2)=b$.

\noindent At the end, let $|y_1|+|x_1-x_2|< a < |y_1|+|y_2|+|x_1-x_2|$. If we choose $v=(x_2,y_2-b)$, if $y_2>0$ or $v=(x_2,y_2+b)$, if $y_2<0$, we have
\[
d^*(v_2,v)=|y_2-(y_2-b)|=b=|y_2-(y_2+b)|.
\]
Since $v\in [v_1,v_2]$, it remains that $d^*(v_1,v)=a$.
\end{proof}

\begin{definition}[\cite{article.1}]
Let $(X, d)$ be a metric space admitting a midpoint map. $(X, d)$ is called ball convex if and only if
\begin{equation}\label{kuglakonveksnost}
d(m(x,y),z)\leq \max \{d(x,z),d(y,z)\}.
\end{equation}
for all $x,y,z\in X$ and for any midpoint $m$ of $(X,d)$. It is called strictly ball convex if the inequality in (\ref{kuglakonveksnost}) is strict whenever $x\neq y$.
\end{definition}

\begin{lemma}[\cite{article.1}]\label{citiranalema}
In geodesic metric space the condition (\ref{kuglakonveksnost}) holds for arbitrary $x,y,z\in X$ if and only if it holds for arbitrary $x,y,z\in X$ for which $d(x,z)=d(y,z)$.
\end{lemma}

\begin{theorem}\label{KuglaKonveksan}
The space $(\rR^2,d^*)$ is strictly ball convex.
\end{theorem}
\begin{proof}
We have already established that $(\rR^2,d^*)$ is a geodesic metric space (\cite{article.11}), so we will use condition of Lemma \ref{citiranalema} for the proof. Let $v_1=(x_1,y_1), v_2=(x_2,y_2), (x,y)\in \rR^2$, $x_1\neq x_2$, such that $(x,y)\in A$, where $A$ is the set of all
points that are equidistant from points $v_1$ and $v_2$, i.e.
\[
d^*((x_1,y_1),(x,y))=d^*((x_2,y_2),(x,y)).
\]
The point $(\overline{x},\overline{y})=\left(\dfrac{|y_2|+x_2-|y_1|+x_1}{2},0\right)$ is the middle of the metric segment $[v_1,v_2]$ (Theorem 2.1 , \cite{article.11}), hence $m(v_1,v_2)=(\overline{x},\overline{y})$. Since $(x,y)\in A$, we have $d^*((\overline{x},\overline{y}),(x,y))=|y-\overline{y}|=|y|$, and from the definition of distance we have
\[
d^*((x_1,y_1),(x,y))=|y_1|+|y|+|x-x_1|,
\]
\[
d^*((x_2,y_2),(x,y))=|y_2|+|y|+|x-x_2|.
\]
It is obvious that following holds
\[
d^*((\overline{x},\overline{y}),(x,y))=|y|\leq \max \{d^*((x_1,y_1),(x,y)),d^*((x_2,y_2),(x,y))\},
\]
and because of the assumption $x_1<x_2$ the strict inequality also holds.

In the case that $x_1=x_2$, the point $(\overline{x},\overline{y})=\left(x_1,\displaystyle \frac{y_1+y_2}{2}\right)$ is the only point of the set $A$ (Theorem 2.2 , \cite{article.11}), and thus the middle of the metric segment, and the condition (\ref{kuglakonveksnost}) is trivially satisfied. Assuming that $v_1\neq v_2$ we have
\[
d^*((\overline{x},\overline{y}),(x,y))=0 < |y-y_1|=|y-y_2|= \max \{d^*((x_1,y_1),(x,y)),d^*((x_2,y_2),(x,y))\},
\]
which proves the theorem.
\end{proof}

\begin{definition}[\cite{article.1}]
Let $(X, d)$ be a metric space admitting a midpoint map. $(X, d)$ is called distance convex if and only if
\begin{equation}\label{distancekonveksnost}
d(m(x,y),z)\leq \dfrac{1}{2}(d(x,z)+d(y,z)).
\end{equation}
for all $x,y,z\in X$ and for any midpoint $m$ of $(X,d)$. It is called strictly distance convex if and only if the inequality
(\ref{distancekonveksnost}) is strict for all $x, y, z \in X$ satisfying $d(x,y) > |d(x,z) - d(y,z)|$.
\end{definition}

Distance convexity implies ball convexity (Proposition 3. \cite{article.1}) and in the general case, the converse is not valid. Under certain assumptions, the converse of the statement will hold. For example, in complete Riemannian manifolds this two convexity are equivalent.

\begin{definition}[\cite{article.1}]\label{Dnpbc}
A geodesic metric space $(X,d)$ is said to be globally non positively Busemann curved (NPBC) if for any three points $x,y,z \in X$ and midpoints $m(x, z)$ and $m(y, z)$ it holds
\begin{equation}\label{npbc}
d(m(x,z),m(y,z))\leq \dfrac{1}{2}d(x,y).
\end{equation}
\end{definition}

\begin{theorem}[\cite{article.1}]\label{Tvezabksadkunpbc}
A global non positively Busemann curved geodesic metric space is (strictly) ball convex if and only if it is (strictly) distance convex.
\end{theorem}

\begin{theorem}
The space $(\rR^2,d^*)$ is NPBC metric space.
\end{theorem}
\begin{proof}
Let $v_1=(x_1,y_1), \ v_2=(x_2,y_2)$ and $v=(x,y)$ be arbitrary points in $\rR^2$. By $m(v_1,v)$ and $m(v_2,v)$ we denote the midpoints of the metric segments $[v_1,v]$ and $[v_2,v]$ respectively. Because of the characteristics of the middle of the metric segment (Theorem 2.1 and Theorem 2.2, \cite{article.11}) we know that the first coordinates of the points $m(v_1,v)$ and $m(v_2,v)$ are either equal to the first coordinates of the points $v_1$ and $v_2$ respectively, or their second coordinates are 0. Therefore, we have,
\[
d^*(v_1,m(v_1,v))+d^*(m(v_1,v),m(v_2,v))+d^*(m(v_2,v),v_2)=d^*(v_1,v_2).
\]
Therefore,
\begin{align*}
d^*(m(v_1,v),m(v_2,v)) & = d^*(v_1,v_2)-(d^*(v_1,m(v_1,v))+d^*(m(v_2,v),v_2)) \\
& = d^*(v_1,v_2)-\left(\dfrac{1}{2}d^*(v_1,v)+\dfrac{1}{2}d^*(v_2,v)\right) \\
& \leq d^*(v_1,v_2)-\dfrac{1}{2}d^*(v_1,v_2)=\dfrac{1}{2}d^*(v_1,v_2).
\end{align*}
Because of the arbitrariness of the points $v_1$, $v_2$ and $v$, according to Definition \ref{Dnpbc}, we conclude  that $(\rR^2,d^*)$ is NPBC space.
\end{proof}

Based on the Theorem \ref{Tvezabksadkunpbc}, as a direct consequence we have the following statement.
\begin{corollary}\label{DistanceKonveksan}
The space $(\rR^2,d^*)$ is distance convex metric space.
\end{corollary}

\begin{definition}[\cite{article.6}]\label{DStriktnaKonveksnost}
We say that the metric space $(X,d)$ is strictly convex if and only if for any three points $x,y,z\in X$, for which $d(x,y)=d(x,z )$, and for an arbitrary point $t$ that is metrically between the points $y$ and $z$, $d(x,t)<d(x,y)$ holds.
\end{definition}
\begin{theorem}\label{TstriktnaKonveksnost}
The space $(\rR^2,d^*)$ is a strictly convex metric space.
\end{theorem}
\begin{proof}
Let $v_1=(x_1,y_1)$, $v_2=(x_2,y_2)$ and $v_3=(x_3,y_3)$ be arbitrary from $\rR^2$ such that $d^*(v_1,v_2 )=d^*(v_1,v_3)$. Assume, without loss of generality, that $x_1<x_2<x_3$. Now let $v=(x,y)\in [v_2,v_3]$ be arbitrary. According to Lemma \ref{MetrickiSegment}, we will consider three cases. \\
\textbf{I:} Let $v\in M_1=\{(x_2,a) \ | \ a\in [0,y_2] (\textrm{ or } a\in [y_2,0])\}$. Then we have,
\[
d^*(v_1,v)=|y_1|+|x_1-x_2|+|a| < |y_1|+|x_1-x_2|+|y_2|=d^*(v_1,v_2).
\]
\noindent \textbf{II:} Let $v\in M_2=\{(x_3,a) \ | \ a\in [0,y_2] (\textrm{ or } a\in [y_2,0])\}$. Then
\[
d^*(v_1,v)=|y_1|+|x_3-x_1|+|a|<|y_1|+|x_3-x_1|+|y_2|=d^*(v_1,v_3)=d^*(v_1,v_2).
\]
\noindent \textbf{III:} If $v \in M_3=\{(b,0) \ | \ x_2\leq b\leq x_3\}$, then
\[
d^*(v_1,v)=|y_1|+|x_1-b|+|0|< |y_1|+|x_1-x_3|+|y_3|=d^*(v_1,v_3)=d^*(v_1,v_2).
\]
So, if $v\in [v_2,v_3]$, then $d^*(v_1,v)<d^*(v_1,v_2)$. According to the Definition \ref{DStriktnaKonveksnost} we conclude that $(\rR^2,d^*)$ is a strictly convex metric space.
\end{proof}

\begin{definition}[\cite{article.5}]
A metric space $(X,d)$ is called an $M$-space if and only if for all $x,y\in X$ with $d(x,y)=\lambda$ and for all $r\in [0,\lambda]$ there exists a unique $z_r\in X$ such that
$$B[x,r]\cap B[y,\lambda -r]=\{ z_r\} \ ,$$
i.e. $d(x,y)=d(x,z_r)+d(z_r,y)$.
\end{definition}

The strict convexity of the space $(\rR^2,d^*)$ additionally means that it is also an M-space (Remarks 1. \cite{article.4}, \cite{article.6}).

\begin{definition}[\cite{article.6}]
Let $(X,d)$ be an M-space. We say that a metric segment in $X$ has a unique prolongations if and only if for any four points $x,y,z,t \in X$ that satisfy the conditions
\begin{align*}
d(x,y)+d(y,z)& =d(x,z), \\
d(x,y)+d(y,t)& = d(x,t), \\
d(y,z)& =d(y,t),
\end{align*}
$z=t$ holds.
\end{definition}

\begin{definition}[\cite{article.6}]
An M-space is said to be externally convex if and only if
\[
(\forall x,y\in X, \ x\neq y)(\exists z\in X)(z\neq y \land d(x,y)+d(y,z)=d(x,z)).
\]
An M-space is said to be strongly externally convex if and only if
\[
(\forall x,y\in X, \ x\neq y, \ d(x,y)=\lambda)(\forall k > \lambda)(\exists_1z\in X) d(x,y)+d(y,z)=d(x,z)=k.
\]
\end{definition}

Now, we will analyze the property of externally convexity, and show that $\rR^2$ is externally convex but is not strongly externally convex. \\
Let $v_1=(x_1,y_1)$ and $v_2=(x_2,y_2)$ be arbitrary points in $\rR^2$, such that $v_1\neq v_2$. According to the definition of the metric, $d^*(v_1,v_2)=|y_1|+|x_2-x_1|+|y_2|$. Without loss of generality, let $y_2>0$. Let us observe the point $v=(x_2,y)$, for arbitrary $y>y_2$. Then $v\neq v_2$ and it is
\[
d^*(v_1,v)=|y_1|+|x_2-x_1|+|y|=|y_1|+|x_2-x_1|+|y_2|+|y-y_2|=d^*(v_1,v_2 )+d^*(v_2,v).
\]
Therefore, the space $(\rR^2,d^*)$ is externally convex.

Let $v_1=(x_1,y_1)$ and $v_2=(x_2,y_2)$ be arbitrary points in $\rR^2$, such that $v_1\neq v_2$. We denote by $\lambda=d^*(v_1,v_2)$. Now let $k>\lambda$. Without loss of generality, let $y_2>0$. Let us observe the points $v'=(x_2,y_2+k-\lambda)$ and $v''=(x_2,y_2-k+\lambda)$. It is valid for them
\begin{align*}
d^*(v_1,v')& =|y_1|+|x_2-x_1|+y_2+k-\lambda=|y_1|+|x_2-x_1|+|y_2|+k-\lambda=k, \\
d^*(v_1,v'')& =|y_1|+|x_2-x_1|+|y_2-k+\lambda|.
\end{align*}
For $d^*(v_1,v'')=k$, it would have to be $|y_2-k+\lambda|=y_2+k-\lambda$, and this will only be the case if $y_2=0$ . This gives us an idea for the next example. Let $v_1=(1,1)$ and $v_2=(3,0)$. Then $d^*(v_1,v_2)=3$. Let $k=4$. Let us consider the points $v'=(3,1)$ and $v''=(3,-1)$. Now we have $d^*=(v_1,v')=4$ and $d^*(v_1,v_2)+d^*(v_2,v')=3+1=4$. Also, $d^*(v_1,v'')=4$ and $d^*(v_1,v_2)+d^*(v_2,v'')=3+1=4$. So, there is no unique point $v\in \rR^2$ such that $d^*(v_1,v_2)+d^*(v_2,v')=d^*(v_1,v)=k$, and the space $(\rR^2,d^*)$ is not strongly externally convex. \\
In normed linear spaces strict convexity implies that metric segments have unique prolongations (Proposition 2, \cite{article.4}). In a strictly convex metric spaces, metric segment need not to have unique prolongations. One of the examples is obviously the space $(\rR^2,d^*)$ which is strictly convex $M$-space whose metric segments don't have unique prolongations.

In the notation and terminology from \cite{article.2}, the results of Theorem \ref{KuglaKonveksan} and Corollary \ref{DistanceKonveksan} mean that  $(\rR^2,d^*)$ is $\infty$-convex, respectively $1$-convex metric space. This again, due to the fact that our observed space is NPBC, means that $(\rR^2,d^*)$ $p$-convex for all $p\in [1,\infty]$ (Corollary 5, \cite{article.2}).

%%%%%%%%%%%%%%%%%%%%%%%%%%%%%%%%%%%%%%%%%%%%%%%%%%%%%%%%%%%%%%%%%%%%%%%%%%%%%%%%%%%%%%%%%%%%%%%%%
%%%%%%%%%%%%%%%%%%%%%%%%%%%%%%%%%%%%%%%%%%%%%%%%%%%%%%%%%%%%%%%%%%%%%%%%%%%%%%%%%%%%%%%%%%%%%%%%%
%%%%%%%%%%%%%%%%%%%%%%%%%%%%%%%%%%%%%%%%%%%%%%%%%%%%%%%%%%%%%%%%%%%%%%%%%%%%%%%%%%%%%%%%%%%%%%%%%
\section{NUC and UC of the space $(\rR^2,d^*)$}

Let $D\subset \mathbb{R}^2$ be a bounded set. Kuratowski measure of noncompactness $\alpha$ is defined as the infimum of all positive numbers $\varepsilon$ such that the set $D$ can be covered by finitely many sets of diameter not larger than $\varepsilon$, and Hausdorff measure of noncompactness $\chi$ is the infimum of all positive numbers $\varepsilon$ such that the set $D$ can be covered by finitely many balls of radius not larger than $\varepsilon$.

\begin{definition}[\cite{bugajewski1998measures}]
Let $D$ be an arbitrary bounded subset in $\mathbb{R}^2$. We say that $y'\in \mathbb{R}$ satisfies:
\begin{enumerate}
   \item condition $A^*(D)$, if for each $y<y'$ there is at least countably many (not finitely many) points $v_n=(x_n,y_n)\in D$ such that
       \[
       x_n\neq x_m, \ \ (n\neq m), \ \ \ y<y_n\leq y', \ \ (n\in \mathbb{N}).
       \]
   \item condition $A_*(D)$ if for each $y>y'$ there is at least countably many (not finitely many) points $v_n=(x_n,y_n)\in D$ such that
       \[
         x_n\neq x_m, \ \ (n\neq m), \ \ \ y>y_n\geq y', \ \ (n\in \mathbb{N}).
       \]
\end{enumerate}
Let us denote that $y^*(D)=\sup\{y'|y' \textrm{ satisfies $A^*(D)$ or $A_*(D)$ } \}$. If there is any number $y'$ that does not satisfy either the condition $A^*(D)$ or the condition $A_*(D)$, we assume that $y^*(D)=0$.
\end{definition}
D. Bugajewski and E. Grzelaczyk have calculated the measure of noncompactness of Kuratowski $\alpha$ on the space $(\rR^2,d^*)$.
\begin{theorem}[Theorem 2., \cite{bugajewski1998measures}]
For an arbitrary bounded set $D\subset \rR^2$ with the river metric, it holds that
\[
\alpha (D)=2y^*(D) \ \textrm{ and } \ \chi(D)=y^*(D).
\]
\end{theorem}

In addition to these two measures of noncompactness of the set, we also consider the Istratescue measure $\beta$, which is defined by
\begin{align*}
\beta(D) & =\inf \{r>0 \ | \ \textrm{$D$ does not have an infinite $r$-separation}\}\\
& =\sup \{r>0 \ | \ \textrm{$D$ has an infinite $r$-separation}\},
\end{align*}
where for the set $D$ we say that it is an $r$-separation (in the complete metric space) if $d(x,y)>r$, for all $x,y\in D$, $x\neq y $.

Let $D\subset \rR^2$ with the river metric $d^*$ and let $y'\in \rR^2$ satisfies the condition $A^*(D)$. This means that there is a sequence $(v_n)_{n\in \mathbb{N}}=((x_n,y_n))_{n\in nN}\subset D$, such that $x_n\neq x_m$ for $n\neq m$ and $y<y_n\leq y '$ for arbitrary $y<y'$. Now we have,
\[
d^*(v_n,v_m)=|y_n|+|y_m|+|x_n-x_m|\geq \max\{|y_n|,|y_m|\}\geq \inf \{|y_n| \ | \ n\in \nN\}=r.
\]
This means that we have $r$-separation and
\[
\beta (D)=\sup \limits_{(x_n,y_n)\in D}\inf \limits_{n\in \nN}|y_n|=y^*(D)=\chi(D).
\]
Some properties of the Kuratowski and Istratescue measure on the space $\mathbb{R}^2$ with the river metric were given in \cite {article.15}. \\
The modulus of noncompact convexity in relation to the measure of noncompactness $\alpha$, on the metric space $(\rR^2,d^*)$ is defined by
\[
\triangle_{\rR^2,\alpha}(\varepsilon)=inf\{1-d^*(0,A) \ | \ A\subseteq B((0,0),1), \ A=coA=\overline{A}, \ \alpha (A)\geq \varepsilon\}.
\]
Let $A \subset \rR^2$ be an arbitrary closed, convex subset of the unit ball, for which $\alpha(A)=2y^*(A)\geq \varepsilon$ holds. Hence
\[
d^*((0,0),A)=\inf\{d^*((0,0),(x,y)) \ | \ (x,y)\in A\}=\inf\{|x|+|y| \ | \ (x,y)\in A\}.
\]
Now we have
\[
1-d^*((0,0),A)=1-\inf\{|x|+|y| \ | \ (x,y)\in A\}=\sup \{1-|x|-|y| \ | \ (x,y)\in A\}.
\]
Therefore,
\[
\triangle_{\rR^2,\alpha}(\varepsilon)=\inf \limits_{\alpha(A)\geq \varepsilon}\sup \limits_{(x,y)\in A}(1-|x|-|y|).
\]
Due to Theorem \ref{Tkonveksniskupovi} we will consider two different cases, considering the fact that the set $A$ is convex. \\

\textbf{I}: Let $A=B((x^*,y^*),r)$, where $|y^*|<r$ (if $|y^*|\geq r$ then the set $A$ is relatively compact, so $\alpha(A)=0$). Then the point $\overline{v}=(\overline{x},0)\in A$ is the closest to the point $(0,0)$,  satisfying the condition
\[
|y^*|+|x^*-\overline{x}|=r.
\]
Hence
\[
\overline{x}=\left\{\begin{array}{lcl}x^*-r+|y^*| & ; & x^*>\overline{x}\\ x^*+r-|y^*| & ; & x^*\leq \overline{x}. \end{array} \right.
\]
Without loss of generality, let the set $A$ be "to the right of $(0,0)$", that is, let $\overline{x}=x^*-r+|y^*|$. Then
\[
d^*((0,0),A)=d^*((0,0),(\overline{x},0))=\overline{x} .
\]
Therefore
\[
\triangle_{\rR^2,\alpha}(\varepsilon)=\inf \limits_{\alpha(A)\geq \varepsilon}(1-\overline{x}).
\]
Since $y^*(A)=y^*+r$, then $\alpha(A)=2(y^*+r)$, so the problem of determining the modulus of noncompact convexity is reduced to solving the problem of minimizing the function $1 -\overline{x}=1-x^*+r-|y^*|$, under the condition $2(y^*+r)\geq \varepsilon$. Since $A\subseteq B((0,0),1)$, without loss of generality we can assume that $y^*<0$. Now, the problem of minimization is defined by
\begin{align*}
1-x^*+r+y^* & \rightarrow \ inf \\
y^*+r & \geq \frac{\varepsilon}{2} \\
0\leq x^*,\, y^* & \leq 1.
\end{align*}
Obviously, the infimum is achieved for $y^*=\displaystyle \frac{\varepsilon}{2}-r$ and $x^*=1$, and we have
\[
\triangle_{\rR^2,\alpha}(\varepsilon)=\frac{\varepsilon}{2}.
\]
\textbf{II}: If $A=[a,b]\times [c,d]$ ($a,b,c,d\in \rR$, $c\leq 0$ and $d\geq 0$), then the point $(a,0)$ is the closest to the point $(0,0)$, i.e.
\[
d^*((0,0),A)=d^*((0,0),(a,0))=|a|.
\]
Now, determining the modulus of noncompact convexity is reduced to the minimization of the function $1-d^*((0,0),A)=1-|a|$, provided that $\alpha(A)\geq \varepsilon$. Since $y^*(A)=d$, we conclude that $\alpha(A)=2d\geq \varepsilon$. So, $d\geq \displaystyle \frac{\varepsilon}{2}$ holds. Therefore, $a\leq b\leq 1-\displaystyle \frac{\varepsilon}{2}$. Hence
\[
\triangle_{\rR^2,\alpha}(\varepsilon)=\inf \limits_{\alpha(A)\geq \varepsilon}(1-|a|)=1-\left(1-\frac{\varepsilon}{2}\right)=\frac{\varepsilon}{2}.
\]

Knowing the connections between modulus of the noncompactness  for different measures of noncompactness, we conclude that
\[
\triangle_{\rR^2,\chi}(\varepsilon)=\varepsilon=\triangle_{\rR^2,\beta}(\varepsilon) .
\]
According to the definition of the characteristic of nearly uniform convexity (shortly NUC)  (\cite{Ayerbe_knjiga}), we have
\[
\varepsilon_{\alpha}(\rR^2)=\sup\{\varepsilon\geq 0 \ | \ \triangle_{\rR^2,\alpha}(\varepsilon)=0\}=0 .
\]
Obviously, $\varepsilon_{\chi}(\rR^2)=\varepsilon_{\beta}(\rR^2)=0$ also holds. Therefore, $(\rR^2,d^*)$ is a NUC space. \\
At the end, we recall the definition of the uniformly convex $W$-convex metric space and we prove this property for the space $(\rR^2,d^*)$.

\begin{definition}[\cite{article.13}]
A $W$-convex metric space is said to be uniformly convex (shortly UC) if for each $\varepsilon >0$ there is a $\delta (\varepsilon)>0$ such that for $c>0$, $x,y,z\in X$ with $d(z,x)\leq c$, $d(z,y)\leq c$ and $d(x,y)\geq c\varepsilon$, we have
\[
d(z,W(x,y,\displaystyle \frac{1}{2}))\leq c(1-\delta) .
\]
\end{definition}

\begin{theorem}
The space $(\rR^2,d^*)$ is uniformly convex space.
\end{theorem}

\begin{proof}
Let $\varepsilon>0$ and $c>0$ be arbitrary, and let $v_1=(x_1,y_1), v_2=(x_2,y_2), z=(z_x,z_y)\in \rR^2$ such that
\begin{align*}
d^*(v_1,z)& =|y_1|+|z_y|+|x_1-z_x|\leq c , \\
d^*(v_2,z)& =|y_2|+|z_y|+|x_2-z_x|\leq c \ \textrm{ i } \\
d^*(v_1,v_2)& \geq c\varepsilon .
\end{align*}
Let
\[
\delta =\min \left \{ \frac{\varepsilon |y_1|}{d*(v_1,v_2)},\frac{\varepsilon |y_2|}{d^*(v_1,v_2)}\right \} .
\]
If
\[
W\left(v_1,v_2, \frac{1}{2}\right)=\left(-\frac{1}{2}|y_1|+\frac{1}{2}|y_2|+\frac{1}{2}x_1+\frac{1}{2}x_2,0\right),
\]
and if $z_x\geq \displaystyle \frac{|y_2|-|y_1|+x_2-x_1}{2}$, then
\begin{align*}
    d^*\left(z,W\left(v_1,v_2,\frac{1}{2}\right)\right) & =  |z_y|+z_x-\frac{|y_2|+x_2-|y_1|+x_1}{2} \\
     & =  \frac{|z_y|+|y_1|+z_x-x_1}{2}+\frac{|z_y|-|y_2|+z_x-x_2}{2} \\
     & \leq   \frac{d^*(v_1,z)}{2}+\frac{d^*(v_2,z)}{2}-|y_2| \\
     & \leq  c-|y_2|=c\left (1-\frac{|y_2|}{c}\right ) \\
     & \leq  c\left (1-\frac{\varepsilon |y_2|}{d^*(v_1,v_2)} \right) \leq c(1-\delta),
\end{align*}
i.e. $(\rR^2,d^*)$ is uniformly convex. The proof is analogous when $z_x\leq \displaystyle \frac{|y_2|-|y_1|+x_2-x_1}{2}$. \\
If
\[
W\left(v_1,v_2,\frac{1}{2}\right)=\left (x_1,\frac{1}{2}y_1-\frac{1}{2}|y_2|-\frac{1}{2}|x_1-x_2|\right ) ,
\]
where $y_1>0$, then let $\delta =\varepsilon$. It holds
\begin{align*}
    d^*\left(z,W\left(v_1,v_2,\frac{1}{2}\right)\right) & = |z_y|+|\frac{1}{2}y_1-\frac{1}{2}|y_2|-\frac{1}{2}|x_1-x_2||+|z_x-x_1| \\
     & =  |z_y|+ \frac{y_1-|y_2|-|x_1-x_2|}{2}+|z_x-x_1| \\
     & =  |z_y|+y_1- \frac{1}{2}(y_1+|y_2|+|x_1-x_2|)+|z_x-x_1| \\
     & =  d^*(v_1,z)-d^*(v_1,v_2) \\
     & \leq  c-c\varepsilon =c(1-\varepsilon)=c(1-\delta) ,
\end{align*}
i.e. $(\rR^2,d^*)$ is uniformly convex. Analogously, we prove other cases, depending on the value of $W\left(v_1,v_2,\frac{1}{2}\right)$.
\end{proof}

%%%%%%%%%%%%%%%%%%%%%%%%%%%%%%%%%%%%%%%%%%%%%%%%%%%%%%%%%%%%%%%%%%%%%%%%%%%%%%%%%%%%%%%%%%%%%%%%%%%%%%%%%%%%%%%%%%%%%%%%%%%%%%%%%%%%%%%%%%%%%

\end{document}